\documentclass[reqno]{amsart}
\usepackage{amsfonts,amsmath,amsthm,amssymb,mathrsfs}
\usepackage{color}
\usepackage{amsmath,amssymb,amsfonts,bm}
\usepackage{graphicx}
\usepackage{newunicodechar}
\usepackage{xcolor}
\numberwithin{equation}{section}
\usepackage[pdfstartview=FitH,colorlinks=true]{hyperref}
\hypersetup{urlcolor=red, linkcolor=blue,citecolor=red}
\allowdisplaybreaks

\theoremstyle{plain}

\newtheorem{thm}{Theorem}[section]

\newtheorem{prop}[thm]{Proposition}

\newtheorem{lemma}[thm]{Lemma}
\newtheorem{remark}[thm]{Remark}

\begin{document}
\title[Landau equation]
{Regularity of the spatially homogenous fractional Kramers-Fokker-Planck equation
}

\author[C.-J. Xu \& Y. Xu]
{Chao-Jiang Xu and Yan Xu}

\address{Chao-Jiang Xu and Yan Xu
\newline\indent
School of Mathematics and Key Laboratory of Mathematical MIIT,
\newline\indent
Nanjing University of Aeronautics and Astronautics, Nanjing 210016, China
}
\email{xuchaojiang@nuaa.edu.cn; xuyan1@nuaa.edu.cn}

\date{\today}



\begin{abstract}
We study the Cauchy problem of the spatially homogenous fractional Kramers-Fokker-Planck equation and show that the solution enjoys Gevrey regularity and decays estimation with an $L^2$ initial datum for positive time.
\end{abstract}

\maketitle

\section{Introduction}
\par The Cauchy problem of fractional Kramers-Fokker-Planck equation reads
\begin{equation*}
\left\{
\begin{aligned}
 &\partial_t u+v\cdot\nabla_x u+\langle v\rangle^\gamma\left((1-\Delta_{v})^s+\langle v\rangle^{2s}\right)u=f(t, x, v),\\
 &u|_{t=0}=u_0,
\end{aligned}
\right.
\end{equation*}
where $u=u(t, x, v)\ge0$ is the density distribution function of particles at time $t\ge0$, and $x, v\in\mathbb R^3$ represent the position and velocity variables of particles, respectively, with $\gamma>-3$, $0<s\le1$.

Here, $s$ and $\gamma$ are parameters that lead to the classical Kramers-Fokker-Planck equation if $\gamma=0, s=1$ and it is the leading term to linear Landau operator of Maxwellian case; if $s=1$, it is the leading term to linear Landau operator of the soft and hard potential case; if $0<s<1$, it is the leading term to linear Boltzmann operator of the soft and hard potential case.

The Kramers equation as a special Fokker-Planck equation was initially derived by Kramers~\cite{K-1} to describe the kinetics of a chemical reaction. Later, it turned out that it had more general applicability, e.g., to different fields such as supersonic conductors, Josephson tunneling junction, relaxation of dipoles, and second-order phase-locked loops. Kolmogorov~\cite{K-2} first studied this equation and gave an explicit formula for the fundamental solution, which implies the existence and smoothness of the solution. It was the main example for motivation to the general theory of H$\rm\ddot{o}$rmander~\cite{H-1} of hypoelliptic equations. The study of hypoelliptic equations often falls back on pseudodifferential operators, which has been addressed in H$\rm\ddot{o}$rmander’s classical results~\cite{H-1, H-2}. Mathematical analysis of the Kramers-Fokker-Planck equation is initially motivated by a trend to equilibrium for confining potentials~\cite{D-1, H-3, V-1}.

For the Kramers-Fokker-Planck equation, Desvillettes and Villani established the explicit decay of any polynomial order $t^{-1}$ in~\cite{D-1}. Exponential decay was shown in~\cite{T-1} and an explicit rate was given in~\cite{H-5}. Spectral properties for some different Fokker-Planck equations have been discussed by Kolokoltsov~\cite{K-4}. The semiclassical resolvent estimates for the Kramers-Fokker-Planck operator have been studied in~\cite{H-6}. The phase space reduction of the one-dimensional Kramers-Fokker-Planck equation was demonstrated in~\cite{K-3}. The Kramers-Fokker-Planck equation with potential has been studied in~\cite{W-1}. An optimal global-in-time $L^{p}-L^{q}$ estimate for solutions to the Kramers-Fokker-Planck equation with short-range potential has been given in~\cite{W-2}.

In this work, we consider the Cauchy problem of the spatially homogenous fractional Kramers-Fokker-Planck equation
\begin{equation}\label{1-2}
\left\{
\begin{aligned}
 &\partial_t u+\langle v\rangle^\gamma\left((1-\Delta_{v})^s+\langle v\rangle^{2s}\right)u=f(t, v),\\
 &u|_{t=0}=u_0,
\end{aligned}
\right.
\end{equation}
with $0<s<1$ and $\gamma+2s>0$. Moreover, the operator $(1-\Delta)^s$, where for any $u\in\mathcal S(\mathbb R^3)$, $(1-\Delta)^s: \mathcal S(\mathbb R^3)\to \mathcal S(\mathbb R^3)$ is defined via the Fourier transform
\begin{equation}\label{fractional Laplace definition}
    \begin{split}
        (1-\Delta)^su=\mathcal F^{-1}\left((1+|\xi|^{2})^{s}\mathcal Fu\right), \quad\xi\in\mathbb R^{3},
    \end{split}
\end{equation}
where $\mathcal F$ and $\mathcal F^{-1}$ are Fourier transform and its inverse, respectively.

Recalling the definition of classes of symbols $S^{m}_{1, 0}$, a $C^{\infty}$ function $a(x, \xi)$ defined in $\mathbb R^{3}\times\mathbb R^{3}$ belongs to the class $S^{m}_{1, 0}$, if there exists a constant $C_{\alpha, \beta}>0$, independent of $x\in\mathbb R^{3}$ and $\xi\in\mathbb R^{3}$, such that
$$\left|\partial^{\alpha}_{x}\partial^{\beta}_{\xi}a(x, \xi)\right|\le C_{\alpha, \beta}(1+|\xi|)^{m-|\beta|}, \quad \forall \alpha, \beta\in\mathbb N^{3},$$
for all $(x, \xi)\in\mathbb R^{3}\times\mathbb R^{3}$. And the class $S^{m}_{1, 0}$ is a Fr${\rm \acute e}$chet space with the semi-norm
$$|a|_{l}^{m}=\max_{|\alpha+\beta|=l}\sup_{x, \xi}(1+|\xi|)^{|\beta|-m}\left|\partial^{\alpha}_{x}\partial^{\beta}_{\xi}a(x, \xi)\right|.$$
So that the operator $(1-\Delta)^s$ can be viewed as a pseudo-differential operator of symbol $(1+|\xi|^{2})^{s}$.

We say that a $C^{\infty}$ function $u\in\mathcal G^{\sigma}(\mathbb R^{3})$, the Gevrey spaces, where $\sigma>0$, if there exists a constant $c_{0}>0$ such that 
\begin{equation*}
    e^{c_{0}|\xi|^{\frac1\sigma}}\hat u\in L^{2}(\mathbb R^{3}),
\end{equation*}
equivalently, for any multi-indices $\alpha\in\mathbb N^{3}$, there exists a constant $C>0$ such that
$$\|\partial^{\alpha}u\|_{L^{2}(\mathbb R^{3})}\le C^{|\alpha|+1}(\alpha!)^{\sigma},$$
the constant $c_{0}=\frac1C$ is the Gevrey radius. Then we introduced the Gelfand-Shilov space $S^{\mu}_{\nu}(\mathbb{R}^3)$. Let $\mu, \nu>0$ and $\mu+\nu\ge 1$, the $C^{\infty}$ function $u\in S^{\mu}_{\nu}(\mathbb{R}^3)$, if there exist $c_{0}, \tilde c_{0}>0$ such that
$$e^{c_{0}|\xi|^{\frac1\mu}}\hat u\in L^{2}(\mathbb R^{3}) \qquad {\rm and}\qquad e^{\tilde c_{0}|v|^{\frac1\nu}}u\in L^{2}(\mathbb R^{3}).$$

We study the Cauchy problem of \eqref{1-2}, with $0<s<1$, $\gamma+2s>0$, and show that the smooth solution to the Cauchy problem \eqref{1-2} with the $L^2(\mathbb R^3)$ initial datum enjoys the Gelfand-Shilov regularity. The main results read as follows.

\begin{thm}\label{thm1}
    For $0<s<1$, $\gamma+2s>0$, for any $T>0$ and $u_0\in L^2(\mathbb R^3)$. Let $u$ be the smooth solution of the Cauchy problem \eqref{1-2} and $f$ satisfies
    \begin{equation}\label{f}
        \left\|\langle D\rangle^{k}f\right\|_{L^2(\mathbb R^3)}\le A^{k+1}k!,\quad t\in]0,T], \ \forall k\in\mathbb N,
    \end{equation}
    then there exists a constant $C>0$ such that for any $k\in\mathbb N$, we have
    \begin{equation}\label{1-4}
        \left\|\langle D\rangle^{2\tilde s k} u(t)\right\|_{L^2(\mathbb R^3)}\le \frac{C^{k+1}}{t^{k}}k!,\quad\forall t\in]0,T],
    \end{equation}
    with $\tilde s=\min\{1/2, s\}$.
\end{thm}

\begin{thm}\label{thm2}
    For $0<s<1$, $\gamma+2s>0$, for any $T>0$ and $u_0\in L^2(\mathbb R^3)$. Let $u$ be the solution of the Cauchy problem \eqref{1-2} and
    $$
    e^{t\langle v\rangle^{\frac{\gamma}{2}+s}}f\in L^{2}(\mathbb R^3), \quad \forall t\ge0.
    $$
    Then there exists a constant $C>0$ such that for any $k\in\mathbb N$, we have
    \begin{equation}\label{1-5}
        \left\|\langle\cdot\rangle^{(\frac{\gamma}{2}+s)k}u(t)\right\|_{L^2(\mathbb R^3)}\le \frac{C^{k+1}}{t^{k}}k!,\quad\forall t\in]0,T].
    \end{equation}
\end{thm}

Remark: We have then proven that the solution of Cauchy problem \eqref{1-2} is belong to Gelfand-Shilov space, i.e.
$$
u(t)\in S^{\frac{1}{2\tilde s}}_{\frac{1}{\frac{\gamma}{2}+s}}(\mathbb{R}^3),\qquad t>0.
$$

This paper is organized as follows. In section 2, we first consider the estimations of the commutator for the operator $\langle D\rangle^{r}$ with $r>0$ in $L^{2}(\mathbb R^{3})$. Then we give the interpolations of Sobolev spaces. Section 3 shows the energy estimation of the Cauchy problem \eqref{1-2}. In section 4, we construct the Gelfand-Shilov regularity to the solution of the Cauchy problem \eqref{1-2}.

\section{Estimations of Commutator and Interpolation}

In the following, the notation $A\lesssim B$ means there exists a constant $C>0$ such that $A\le C B$, and the notation $[T_1, T_2]$ means $T_1T_2-T_2T_1$, which denotes the commutator. For simplicity, we denote the weighted Lebesgue spaces and the weighted Sobolev spaces
\begin{equation*}
    \left\|\langle\cdot\rangle^\gamma f\right\|_{L^p(\mathbb R^3)}=\|f\|_{p, \gamma},\quad 1\le p\le\infty,
\end{equation*}
\begin{equation*}
    \left\|\langle\cdot\rangle^\gamma f\right\|_{H^m(\mathbb R^3)}=\|f\|_{H^m_\gamma(\mathbb R^3)},\quad m\in\mathbb R,
\end{equation*}
with $\gamma\in\mathbb R$, where we use the notation $\langle v\rangle=(1+|v|^2)^{\frac12}$.

For later use, we need the following estimation of the commutator for the operator $\langle D\rangle^{r}=(1-\Delta)^{\frac{r}{2}}$ with $r>0$.

\begin{lemma}\label{lemma2.2}
     Let $u\in\mathcal S(\mathbb R^{3})$, then for all $m\in\mathbb R$ and $r>0$, there exists a constant $C_{1}>0$, depends on $m$ and $r$, such that
      $$\left\|\left[(1-\Delta)^{\frac{r}{2}}, \langle \cdot\rangle^{m}\right]u\right\|_{L^{2}(\mathbb R^{3})}\le C_{1}\|u\|_{H^{r-1}_{m}}.$$
\end{lemma}
\begin{proof}
     Let $\langle v\rangle^{m}u=g$, then we have
     $$\left[(1-\Delta)^{\frac{r}{2}}, \langle v\rangle^{m}\right]u=(1-\Delta)^{\frac{r}{2}}g-\langle v\rangle^{m}(1-\Delta)^{\frac{r}{2}}\left(\langle v\rangle^{-m}g\right).$$
     Since $(1-\Delta)^{\frac{r}{2}}$ can be viewed as a pseudo-differential operator of symbol $(1+|\xi|^{2})^{\frac{r}{2}}$, it follows that for any $N\in\mathbb N_{+}$,
     \begin{equation*}
         \langle v\rangle^{m}(1-\Delta)^{\frac{r}{2}}\langle v\rangle^{-m}=\sum_{0\le|\alpha|<N}\frac{1}{\alpha!}\langle v\rangle^{m}D_{v}^{\alpha}\langle v\rangle^{-m}a^{(\alpha)}(D_{v})+\langle v\rangle^{m}r_{N}(v,D_{v}),
     \end{equation*}
     where $a^{(\alpha)}(\xi)=\partial^{\alpha}_{\xi}\langle\xi\rangle^{r}$, and
     $$r_{N}(v,\xi)=N\sum_{|\alpha|=N}\int_{0}^{1}\frac{(1-\theta)^{N-1}}{\alpha!}r_{\theta, \alpha}(v, \xi)d\theta,$$
     with $r_{\theta, \alpha}(v, \xi)$ is the oscillating integral, defined via
     $$r_{\theta, \alpha}(v, \xi)=Os-\iint e^{-v'\eta}D^{\alpha}\langle v+v'\rangle^{-m}a^{(\alpha)}(\xi+\theta\eta)\frac{dv'd\eta}{(2\pi)^{3}}.$$
     Using the identity
     $$e^{-iv'\eta}=\langle v'\rangle^{-2l '}\left(1-\Delta_{\eta}\right)^{l'}e^{-iv'\eta}=\langle\eta\rangle^{-2l}\left(1-\Delta_{v'}\right)^{l}e^{-iv'\eta},$$
     then from the fact $2l'>|m|+N+|\beta|+3$, $2l>r+N+|\beta'|+3$, integration by parts and using the Leibniz formula, we can deduce that
     \begin{equation*}
        \begin{split}
             &D^{\beta}_{v}\partial^{\beta'}_{\xi}\left(\langle v\rangle^{m}r_{\theta,\alpha}(v,\xi)\right)\\
             &=\sum_{\beta_{1}+\beta_{2}=\beta}C_{\beta}^{\beta_{1}}D^{\beta_{1}}_{v}\langle v\rangle^{m}\int_{\mathbb R^{3}}(1-\Delta_{\eta})^{l'}a^{(\alpha+\beta')}(\xi+\theta\eta)G(v, \eta)\frac{d\eta}{\langle \eta\rangle^{2l}(2\pi)^{3}},
        \end{split}
    \end{equation*}
    where
    $$G(v, \eta)=\int_{\mathbb R^{3}}e^{-iv'\eta}(1-\Delta_{v'})^{l}\left(\langle v'\rangle^{-2l'}D^{\alpha+\beta_{2}}\langle v+v'\rangle^{-m}\right)d v '.$$
    By using the Leibniz formula and the fact $2l'>|m|+N+|\beta|+3$, one gets
    \begin{equation*}
        \begin{split}
            |G(v, \eta)|&\le\int_{\mathbb R_{3}}\left|(1-\Delta_{v'})^{l}\left(\langle v'\rangle^{-2l'}D^{\alpha+\beta_{2}}\langle v+v'\rangle^{-m}\right)\right|dv'\\
            &\le c_{1}\sum_{|\sigma|\le 2l}\int_{\mathbb R_{3}}\left|\partial^{\sigma}_{v'}\left(\langle v'\rangle^{-2l'}D^{\alpha+\beta_{2}}\langle v+v'\rangle^{-m}\right)\right|dv'\\
            &\le c_{2}\langle v\rangle^{-m-N-|\beta_{2}|}\int_{\mathbb R_{3}}\langle v'\rangle^{-2l'+|m|+N+|\beta_{2}|}dv'\le\tilde c_{2}\langle v\rangle^{-m-N-|\beta_{2}|},
        \end{split}
    \end{equation*}
    with the constant $\tilde c_{2}$ depends on $m$.
Substituting it into $D^{\beta}_{v}\partial^{\beta'}_{\xi}\left(\langle v\rangle^{m}r_{\theta, \alpha}(v, \xi)\right)$, since $\theta\in]0,1[$, by using Peetre's inequality, we have for all $v\in\mathbb R^{3}$
    \begin{equation}\label{r_theta alpha}
        \begin{split}
             &\left|D^{\beta}_{v}\partial^{\beta'}_{\xi}\left(\langle v\rangle^{m}r_{\theta,\alpha}(v,\xi)\right)\right|\le\tilde c_{2}\int_{\mathbb R^{3}}\langle \eta\rangle^{-2l}\langle\xi+\theta\eta\rangle^{r-N-|\beta'|}d\eta\\
             &\le c_{3}\langle\xi\rangle^{r-N-|\beta'|}\int_{\mathbb R^{3}}\langle \eta\rangle^{-2l+|r|+N+|\beta'|}d\eta\le\tilde c_{3}\langle\xi\rangle^{r-N-|\beta'|},
        \end{split}
    \end{equation}
    here we use the fact $2l>r+N+|\beta'|+3$, and the constant $\tilde c_{3}$ depends on $m$, $r$. So that we can obtain 
    $$\langle v\rangle^{m}r_{N}(v, D)\in\Psi_{1,0}^{r-N}.$$
    And therefore, $\langle v\rangle^{m}r_{N}(v, D)(1-\Delta)^{\frac{N-r}{2}}\in\Psi_{1,0}^{0}$, then it follows that
    \begin{equation*}
        \begin{split}
            \left\|[(1-\Delta)^{\frac{r}{2}}, \langle \cdot\rangle^{m}]u\right\|_{L^{2}}&\lesssim\sum_{j=1}^{N-1}\left\|(1-\Delta)^{\frac{r-j}{2}}g\right\|_{L^{2}}+\left\|\langle\cdot\rangle^{m}r_{N}(\cdot,D)g\right\|_{L^{2}}\\
            &\lesssim\sum_{j=1}^{N-1}\left\|g\right\|_{H^{r-j}}+\left\|g\right\|_{H^{r-N}}\le C_{1}\left\|u\right\|_{H^{r-1}_{m}},
        \end{split}
    \end{equation*}
    with $C_{1}$ depends on $m$ and $r$.
\end{proof}

\begin{remark}
Taking $r=2s$ with $0<s\le1/2$, we have
\begin{equation}\label{0-s-1/2}
        \begin{split}
            \left\|[(1-\Delta)^{s}, \langle \cdot\rangle^{m}]u\right\|_{L^{2}}\le C_{2}\|u\|_{2, m},
        \end{split}
    \end{equation}
    meanwhile, for $1/2<s<1$, we have
    \begin{equation}\label{1/2-s-1}
        \begin{split}
            \left\|[(1-\Delta)^{s}, \langle \cdot\rangle^{m}]u\right\|_{L^{2}}\le C_{3}\|u\|_{H^{2s-1}_{m}},
        \end{split}
    \end{equation}
    with $C_{2}$, $C_{3}$ depend on $m$ and $r$.
\end{remark}

And we also need the following interpolations. Firstly, we give the compactness (Aubin-Lions) lemma, which plays an important role in the proof of Lemma \ref{interpolation}.
\begin{lemma}(~\cite{S-1})\label{Lions}
     Let $E_{1}, E_{2}, E_{3}$ be the Banach spaces satisfy
     $$E_{1}\hookrightarrow E_{2}\hookrightarrow E_{3},$$
     and the embedding $E_{1}\hookrightarrow E_{3}$ is compact, then the embedding $E_{1}\hookrightarrow E_{2}$ is compact if and only if for any $\epsilon>0$, there exist a constant $C_{\epsilon}>0$, depends on $\epsilon$, such that
     $$\|x\|_{E_{2}}\le\epsilon\|x\|_{E_{1}}+C_{\epsilon}\|x\|_{E_{3}},\quad x\in E_{1}.$$
\end{lemma}

\begin{lemma}\label{interpolation}
     For $u\in\mathcal S(\mathbb R^{3})$. Let $k,l>0$ and $\delta>0$, then for any $\epsilon>0$, there exist a constant $C_{\epsilon}>0$, depends on $\epsilon$, such that
     $$\|u\|_{H^{k}_{l}(\mathbb R^{3})}\le\epsilon\|u\|_{H^{k+\delta}_{l}(\mathbb R^{3})}+C_{\epsilon}\|u\|_{L^{2}(\mathbb R^{3})}.$$
\end{lemma}
\begin{proof}
    From Lemma \ref{Lions}, the key idea of the proof is for any $m>0$, the embedding $H^{m}_{l}(\mathbb R^{3})$ to $L^{2}(\mathbb R^{3})$ is compact.
    For any $\epsilon>0$, choose the constant $R>0$ sufficiently large such that
    \begin{equation}\label{2-4-1}
        (1+|v|^{2})^{-l}\le\epsilon,\quad|v|\ge R/2.
    \end{equation}
    Fixing a cutoff function $\phi\in C^{\infty}(\mathbb R^{3})$, with the properties that $0\le\phi(v)\le1$, and
\begin{equation*}
\phi(v)=\left\{
\begin{aligned}
 &1, \quad v\in B_{R/2}(0),\\
 &0, \quad v\in B^{c}_{R}(0),
\end{aligned}
\right.
\end{equation*}
here $B_{R/2}(0)=\{v: |v|\le1\}$ and $B^{c}_{R}(0)=\{v: |v|>1\}$. From the properties of the Fourier transform, one has $\langle\xi\rangle^{m}\hat\phi\in L^{1}(\mathbb R^{3})$.

    Then by using the Plancherel Theorem and Young inequality, it follows that
    \begin{equation}\label{2-4-2}
       \begin{split}
           &\|\phi u\|_{H^{m}(\mathbb R^{3})}=\left\|\langle\cdot\rangle^{m}\left(\hat\phi*\hat u\right)\right\|_{L^{2}(\mathbb R^{3})}\\
           &\le2^{m}\left\|\left(\langle\cdot\rangle^{m}\hat\phi\right)*\left(\langle\cdot\rangle^{m}\hat u\right)\right\|_{L^{2}(\mathbb R^{3})}\lesssim\left\|\langle\cdot\rangle^{m}\hat u\right\|_{L^{2}(\mathbb R^{3})}=\|u\|_{H^{m}(\mathbb R^{3})}.
        \end{split}
    \end{equation}
    Assume that $\{\langle v\rangle^{l}u_{j}\}$ is a bounded sequence in $H^{m}(\mathbb R^{3})$, then by using the Sobolev inequality, one gets $\{u_{j}\}$ is bounded in $L^{2}(\mathbb R^{3})$. Since the embedding $H^{m}(B_{R}(0))$ to $L^{2}(B_{R}(0))$ is compact, by the inequality \eqref{2-4-2} we can obtain $\{\phi\langle v\rangle^{l}u_{j}\}$ is bounded in $H^{m}(B_{R}(0))$, which implies the subsequence $\{\phi\langle v\rangle^{l}u_{j'}\}\subset\{\phi\langle v\rangle^{l}u_{j}\}$ converges, that is for any $\epsilon>0$, there exists the large $N\in\mathbb N$ such that
    $$\left\|\phi\langle\cdot\rangle^{l}(u_{j'}-u_{n'})\right\|_{L^{2}(B_{R}(0))}\le\epsilon,\quad\forall j', n'\ge N.$$
    And therefore from $\eqref{2-4-1}$, we can conclude
    \begin{equation*}
        \begin{split}
             \|u_{j'}-u_{n'}\|_{L^{2}(\mathbb R^{3})}
             &\le\left\|\phi(u_{j'}-u_{n'})\right\|_{L^{2}(\mathbb R^{3})}+\left\|(1-\phi)(u_{j'}-u_{n'})\right\|_{L^{2}(\mathbb R^{3})}\\
             &\le\left\|\phi\langle\cdot\rangle^{l}(u_{j'}-u_{n'})\right\|_{L^{2}(B_{R}(0))}+\|u_{j'}-u_{n'}\|_{L^{2}(|v|\ge R/2)}\\
             &\le\epsilon+\left\|\langle\cdot\rangle^{-l}\langle\cdot\rangle^{l}(u_{j'}-u_{n'})\right\|_{L^{2}(|v|\ge R/2)}\\
             &\le\epsilon+\epsilon\left(\left\|\langle\cdot\rangle^{l}u_{j'}\right\|_{H^{m}(\mathbb R^{3})}+\left\|\langle\cdot\rangle^{l}u_{n'}\right\|_{H^{m}(\mathbb R^{3})}\right)\lesssim\epsilon,
        \end{split}
    \end{equation*}
    which implies $\{u_{j'}\}$ is a Cauchy sequence in $L^{2}(\mathbb R^{3})$, then by the completeness of $L^{2}(\mathbb R^{3})$ we can obtain that $\{u_{j'}\}$ converges in $L^{2}(\mathbb R^{3})$. Thus for any $m>0$, the embedding $H^{m}_{l}(\mathbb R^{3})$ to $L^{2}(\mathbb R^{3})$ is compact.

     Finally, let $m=k$ and $m=k+\delta$ respectively, then from Lemma \ref{Lions}, it follows that for any $\epsilon>0$,
    $$\|u\|_{H^{k}_{l}(\mathbb R^{3})}\le\epsilon\|u\|_{H^{k+\delta}_{l}(\mathbb R^{3})}+C_{\epsilon}\|u\|_{L^{2}(\mathbb R^{3})}.$$
\end{proof}

The following interpolation has been given in~\cite{H-4}.
\begin{lemma}(~\cite{H-4})\label{interpolation1}
    Let $k,l\in\mathbb R$ and $\delta>0$, then there exists a constant $C(k,l,\delta)$ such that for any $u\in\mathcal S(\mathbb R^{3})$,
    $$\|u\|^{2}_{H^{k}_{l}(\mathbb R^{3})}\le C(k,l,\delta)\|u\|_{H^{k+\delta}_{2l}(\mathbb R^{3})}\|u\|_{H^{k-\delta}(\mathbb R^{3})}.$$
\end{lemma}

\section{Energy Estimates}

In this section, we study the energy estimates of the solution to the Cauchy problem \eqref{1-2}.

\begin{lemma}\label{lemma3.1}
    For $0<s<1$, $\gamma+2s>0$ and $T>0$. Let $u$ is the smooth solution of the Cauchy problem \eqref{1-2}. Assume $u_0\in L^2(\mathbb R^3)$ and $f$ satisfies \eqref{f}. Then there exists a constant $B_{0}>0$, depends on $s, \gamma$ and $T$, such that for any $t\in]0,T]$,
    \begin{equation}\label{3-1}
    \begin{split}
        \left\|u(t)\right\|^2_{L^2(\mathbb R^3)}+\int_0^t\left\|u(\tau)\right\|^2_{H^{s}_{\gamma/2}}d\tau+\int_{0}^{t}\left\|u(\tau)\right\|^{2}_{2,\gamma/2+s}d\tau\le B_{0}.
    \end{split}
    \end{equation}
\end{lemma}
\begin{proof}
     Since $u$ is the solution of the Cauchy problem \eqref{1-2}, we have
     \begin{equation}\label{3-1-1}
         \frac12\frac{d}{dt}\left\|u(t)\right\|^{2}_{L^{2}(\mathbb R^3)}+\left(\langle v\rangle^{\gamma}(1-\Delta)^{s}u, u\right)+\left\|\langle\cdot\rangle^{\gamma/2+s}u(t)\right\|^{2}_{L^{2}(\mathbb R^3)}=(f, u).
     \end{equation}
     Noting that
     \begin{equation*}
        \begin{split}
             &\left(\langle v\rangle^{\gamma}(1-\Delta)^{s}u, u\right)\\
             &=\left((1-\Delta)^{s}\left(\langle v\rangle^{\gamma/2}u\right), \langle v\rangle^{\gamma/2}u\right)+\left(\left[\langle v\rangle^{\gamma/2}, (1-\Delta)^{s}\right]u,\langle v\rangle^{\gamma/2}u\right),
        \end{split}
    \end{equation*}
     since $(1-\Delta)^{s}\in\Psi_{1,0}^{2s}$, from using the G${\rm\mathring{a}}$rding inequality we can get that there exists a constant $C_{0}>0$,
     $$\left((1-\Delta)^{s}\left(\langle v\rangle^{\gamma/2}u\right), \langle v\rangle^{\gamma/2}u\right)\ge\frac12\left\|u(t)\right\|^{2}_{H^{s}_{\gamma/2}(\mathbb R^3)}-C_{0}\left\|u(t)\right\|^{2}_{2, \gamma/2}.$$
     Firstly, for $0<s\le1/2$, using Cauchy-Schwarz inequality and \eqref{0-s-1/2} we get
     \begin{equation*}
        \begin{split}
             \left|\left(\left[\langle v\rangle^{\gamma/2}, (1-\Delta)^{s}\right]u,\langle v\rangle^{\gamma/2}u\right)\right|\le C_{2}\left\|u(t)\right\|^{2}_{2,\gamma/2},
        \end{split}
    \end{equation*}
    substituting it into \eqref{3-1-1}, since $\gamma+2s>0$, from Cauchy-Schwarz inequality, 
     \begin{equation}\label{3-1-2}
     \begin{split}
         &\frac{d}{dt}\left\|u(t)\right\|^{2}_{L^{2}(\mathbb R^3)}+\left\|u(t)\right\|^{2}_{H^{s}_{\gamma/2}(\mathbb R^3)}+\left\|u(t)\right\|^{2}_{2,\gamma/2+s}\\
         &\le\tilde C_{0}\left\|u(t)\right\|^{2}_{2,\gamma/2}+\left\|f(t)\right\|^{2}_{L^{2}(\mathbb R^3)},
     \end{split}
     \end{equation}
     with $\tilde C_{0}$ depends on $C_{0}$ and $C_{2}$. For $\gamma>0$, since $s>0$,  we have 
     $$\tilde C_{0}\left\|u(t)\right\|^{2}_{2,\gamma/2}\le\tilde C_{0}\left\|u(t)\right\|^{2}_{H^{s/2}_{\gamma/2}(\mathbb R^3)},$$ 
     then from Lemma \ref{interpolation} and taking $\epsilon=\frac{1}{2\sqrt{\tilde C_{0}}}$, we obtain that 
     \begin{equation*}
     \begin{split}
         &\tilde C_{0}\left\|u(t)\right\|^{2}_{2,\gamma/2}\le\tilde C_{0}\left\|u(t)\right\|^{2}_{H^{s/2}_{\gamma/2}(\mathbb R^3)}\le\frac12\left\|u(t)\right\|^{2}_{H^{s}_{\gamma/2}(\mathbb R^3)}+\tilde C_{2}\left\|u(t)\right\|^{2}_{L^{2}(\mathbb R^3)},
     \end{split}    
     \end{equation*}
     and for $-2s<\gamma\le0$, by using Lemma \ref{interpolation1}, it follows that $\tilde C_{0}\left\|u(t)\right\|^{2}_{2,\gamma/2}$ also can be bounded by
      \begin{equation*}
         \frac12\left\|u(t)\right\|^{2}_{H^{s}_{\gamma/2}(\mathbb R^3)}+\tilde C_{2}\left\|u(t)\right\|^{2}_{L^{2}(\mathbb R^3)},
     \end{equation*}
    here $\tilde C_{2}$ depends on $C_{0}$ and $C_{2}$. Plugging these back into \eqref{3-1-2}, we have for all $\gamma>-2s$ and $0<s\le1/2$,
    \begin{equation*}
     \begin{split}
         &\frac{d}{dt}\left\|u(t)\right\|^{2}_{L^{2}(\mathbb R^3)}+\left\|u(t)\right\|^{2}_{H^{s}_{\gamma/2}(\mathbb R^3)}+\left\|u(t)\right\|^{2}_{2,\gamma/2+s}\le2\tilde C_{2}\left\|u(t)\right\|^{2}_{L^{2}(\mathbb R^3)}+2\left\|f(t)\right\|^{2}_{L^{2}(\mathbb R^3)},
     \end{split}
     \end{equation*}

    Next, consider the case of $1/2<s<1$, applying \eqref{1/2-s-1} and Cauchy-Schwarz inequality, we have
     $$\left|\left(\left[\langle v\rangle^{\gamma/2}, (1-\Delta)^{s}\right]u,\langle v\rangle^{\gamma/2}u\right)\right|\le C_{3}\left\|u(t)\right\|^{2}_{H^{2s-1}_{\gamma/2}(\mathbb R^3)},$$
     substituting it into \eqref{3-1-1}, by using the fact $\gamma+2s>0$  and Cauchy-Schwarz inequality, one has for all $1/2<s<1$
     \begin{equation}\label{3-1-4}
     \begin{split}
         &\frac{d}{dt}\left\|u(t)\right\|^{2}_{L^{2}(\mathbb R^3)}+\left\|u(t)\right\|^{2}_{H^{s}_{\gamma/2}(\mathbb R^3)}+\left\|u(t)\right\|^{2}_{2,\gamma/2+s}\\
         &\le\tilde C'_{0}\left\|u(t)\right\|^{2}_{H^{2s-1}_{\gamma/2}(\mathbb R^3)}+\left\|f(t)\right\|^{2}_{L^{2}(\mathbb R^3)},
     \end{split}
     \end{equation}
     with $\tilde C'_{0}$ depends on $C_{0}$ and $C_{3}$. Since $0<s<1$, it follows that $0<2s-1<s$. Then as the argument in the case $0<s\le1/2$, if $\gamma>0$, we use Lemma \ref{interpolation}, on the other hand, for $-2s<\gamma\le0$, we use Lemma \ref{interpolation1}. Thus for all $\gamma>-2s$, we can get that $\tilde C'_{0}\left\|u(t)\right\|^{2}_{H^{2s-1}_{\gamma/2}(\mathbb R^3)}$ is bounded by
     \begin{equation*}
         \frac12\left\|u(t)\right\|^{2}_{H^{s}_{\gamma/2}(\mathbb R^3)}+\tilde C_{3}\left\|u(t)\right\|^{2}_{L^{2}(\mathbb R^3)},
     \end{equation*}
     with $\tilde C_{3}$ depends on $C_{0}$ and $C_{3}$. Plugging it back into \eqref{3-1-4}, we can get
     \begin{equation*}
     \begin{split}
         &\frac{d}{dt}\|u(t)\|^{2}_{L^{2}(\mathbb R^3)}+\|u(t)\|^{2}_{H^{s}_{\gamma/2}(\mathbb R^3)}+\|u(t)\|^{2}_{2,\gamma/2+s}\le2\tilde C_{3}\|u(t)\|^{2}_{L^{2}(\mathbb R^3)}+2\|f(t)\|^{2}_{L^{2}(\mathbb R^3)}.
     \end{split}
     \end{equation*}
     Thus, for all $\gamma>-2s$ and $0<s<1$, we have
     \begin{equation*}
     \begin{split}
         &\frac{d}{dt}\left\|u(t)\right\|^{2}_{L^{2}(\mathbb R^3)}+\left\|u(t)\right\|^{2}_{H^{s}_{\gamma/2}(\mathbb R^3)}+\left\|u(t)\right\|^{2}_{2,\gamma/2+s}\\
         &\le\max\left\{2\tilde C_{2}, 2\tilde C_{3}\right\}\left\|u(t)\right\|^{2}_{L^{2}(\mathbb R^3)}+2\left\|f(t)\right\|^{2}_{L^{2}(\mathbb R^3)},
     \end{split}
     \end{equation*}
     then applying Gronwall inequality and \eqref{f}, taking
     $$B_{0}=2A^{2}T^{2}e^{T\max\left\{2\tilde C_{2}, 2\tilde C_{3}\right\}}+2A^{2}T,$$
     it follows that for any $t\in]0,T]$,
     \begin{equation*}
     \begin{split}
         &\left\|u(t)\right\|^{2}_{L^{2}(\mathbb R^{3})}+\int_{0}^{t}\left\|u(\tau)\right\|^{2}_{H^{s}_{\gamma/2}(\mathbb R^{3})}d\tau+\int_{0}^{t}\left\|u(\tau)\right\|^{2}_{2,\gamma/2+s}d\tau\le B_{0},
     \end{split}
     \end{equation*}
     with $B_{0}>0$, depends on $s$, $\gamma$ and $T$.
\end{proof}

\section{Proof of Theorem \ref{thm1} and Theorem \ref{thm2}}

In this section, we will show the Gelfand-Shilov regularity to the solution of the Cauchy problem \eqref{1-2}. We construct the following estimates, which imply Theorem \ref{thm1} and Theorem \ref{thm2}. 

\begin{prop}\label{prop 4.1}
    For $0<s<1$, $T>0$ and $\gamma+2s>0$. Let $u$ be the smooth solution of Cauchy problem \eqref{1-2}, assume $u_0\in L^2(\mathbb R^3)$ and $f$ satisfies \eqref{f}. Then there exists a constant $B_{1}>0$, such that for any $t\in]0,T]$ and $k\in\mathbb N$,
    \begin{equation}\label{4-1}
    \begin{split}
        &\left\|\left(t\langle D\rangle^{2\tilde s}\right)^{k}u\right\|^2_{(L^{\infty}]0,T];L^2)}+\int_0^T\left\|\left(t\langle D\rangle^{2\tilde s}\right)^ku\right\|^2_{H^s_{\gamma/2}}dt\\
        &\quad+\int_0^T\left\|\left(t\langle D\rangle^{2\tilde s}\right)^ku\right\|^2_{2,\gamma/2+s}dt\le\left(B_{1}^{k+1}k!\right)^2,
    \end{split}
    \end{equation}
    with $B_{1}$ depends on $\gamma$, $s$ and $T$.
\end{prop}
\begin{proof}
    Since $u\in C^{\infty}(\mathbb R^{3})$ is the solution of Cauchy problem \eqref{1-2}, we have
    \begin{equation*}
    \begin{split}
        &\frac{d}{dt}\left(\left(t\langle D\rangle^{2\tilde s}\right)^{k}u\right)+\langle v\rangle^{\gamma}\left((1-\Delta)^{s}+\langle v\rangle^{2s}\right)\left(t\langle D\rangle^{2\tilde s}\right)^{k}u\\
        &=kt^{k-1}\langle D\rangle^{2\tilde sk}u+\left(t\langle D\rangle^{2\tilde s}\right)^{k}f+\left[\langle v\rangle^{\gamma}\left((1-\Delta)^{s}+\langle v\rangle^{2s}\right), \left(t\langle D\rangle^{2\tilde s}\right)^{k}\right]u,
    \end{split}
    \end{equation*}
    taking the scalar product with respect to $\left(t\langle D\rangle^{2\tilde s}\right)^{k}u$, then by using G${\rm\mathring{a}}$rding inequality and Cauchy-Schwarz inequality, one has  
    \begin{equation}\label{dt}
    \begin{split}
        &\frac{d}{dt}\left\|\left(t\langle D\rangle^{2\tilde s}\right)^{k}u\right\|^{2}_{L^{2}}+\left\|\left(t\langle D\rangle^{2\tilde s}\right)^{k}u\right\|^{2}_{H^{s}_{\gamma/2}}+\left\|\left(t\langle D\rangle^{2\tilde s}\right)^{k}u\right\|^{2}_{2,\gamma/2+s}\\
        &\le 2kt^{2k-1}\left\|\langle D\rangle^{2\tilde sk}u\right\|^{2}_{L^{2}}+2C_{0}\left\|\left(t\langle D\rangle^{2\tilde s}\right)^{k}u\right\|^{2}_{2,\gamma/2}+\left\|\left(t\langle D\rangle^{2\tilde s}\right)^{k}f\right\|^{2}_{L^{2}}\\
        &\quad+2\left(\left[(1-\Delta)^{s}, \langle v\rangle^{\gamma/2}\right]\left(t\langle D\rangle^{2\tilde s}\right)^{k}u, \langle v\rangle^{\gamma/2}\left(t\langle D\rangle^{2\tilde s}\right)^{k}u\right)\\
        &\quad+2\left(\left[\langle v\rangle^{\gamma+2s}, \left(t\langle D\rangle^{2\tilde s}\right)^{k}\right]u, \left(t\langle D\rangle^{2\tilde s}\right)^{k}u\right)\\
        &\quad+2\left(\left[\langle v\rangle^{\gamma}(1-\Delta)^{s}, (t\langle D\rangle^{2\tilde s})^{k}\right]u, \left(t\langle D\rangle^{2\tilde s}\right)^{k}u\right)\\
        &=2kt^{2k-1}\left\|\langle D\rangle^{2\tilde sk}u\right\|^{2}_{L^{2}}+2C_{0}\left\|\left(t\langle D\rangle^{2\tilde s}\right)^{k}u\right\|^{2}_{2,\gamma/2}+\left\|\left(t\langle D\rangle^{2\tilde s}\right)^{k}f\right\|^{2}_{L^{2}}\\
        &\quad+I_{1}+I_{2}+I_{3}.
    \end{split}
    \end{equation}

    We show \eqref{4-1} is true by induction on the index $k$. For $k=0$, it is enough to take in \eqref{3-1}. Assume $k\ge1$ and \eqref{4-1} holds true for $0\le m\le k-1$,
    \begin{equation}\label{m}
    \begin{split}
        &\left\|\left(t\langle D\rangle^{2\tilde s}\right)^{m}u\right\|^2_{(L^{\infty}]0,T];L^2)}+\int_0^t\left\|\left(t\langle D\rangle^{2\tilde s}\right)^mu\right\|^2_{H^s_{\gamma/2}}d\tau\\
        &\quad+\int_0^T\left\|\left(t\langle D\rangle^{2\tilde s}\right)^mu\right\|^2_{2, \gamma/2+s}dt\le\left(B_{1}^{m+1}m!\right)^2.
    \end{split}
    \end{equation}
    Now, we prove that \eqref{m} is true for $m=k$.

    For $I_{1}$, we restrict to $0<s\le1/2$ first. By applying \eqref{0-s-1/2} and Cauchy-Schwarz inequality gets
     $$|I_{1}|\le 2C_{2}\left\|\left(t\langle D\rangle^{2\tilde s}\right)^{k}u\right\|^{2}_{2,\gamma/2}.$$
    Then we consider the case $1/2<s<1$. Using \eqref{1/2-s-1} and Cauchy-Schwarz inequality, 
   $$|I_{1}|\le2C_{3}\left\|\left(t\langle D\rangle^{2\tilde s}\right)^{k}u\right\|_{H^{2s-1}_{\gamma/2}}\left\|\left(t\langle D\rangle^{2\tilde s}\right)^{k}\right\|_{2,\gamma/2}\le 2C_{3}\left\|\left(t\langle D\rangle^{2\tilde s}\right)^{k}u\right\|^{2}_{H^{2s-1}_{\gamma/2}}.$$
    For the above two cases, as the argument in Lemma \ref{lemma3.1}, if $\gamma>0$, we apply Lemma \ref{interpolation} and take $\epsilon=\frac{1}{4\sqrt{2C_{2}}}$, then it follows that
    \begin{equation}\label{2-gamma}
    \begin{split}
        &2C_{2}\left\|\left(t\langle D\rangle^{2\tilde s}\right)^{k}u\right\|^{2}_{2,\gamma/2}\le2C_{2}\left\|\left(t\langle D\rangle^{2\tilde s}\right)^{k}u\right\|^{2}_{H^{s/2}_{\gamma/2}}\\
        &\le\frac18\left\|\left(t\langle D\rangle^{2\tilde s}\right)^{k}u\right\|^{2}_{H^{s}_{\gamma/2}(\mathbb R^3)}+\tilde C_{2}\left\|\left(t\langle D\rangle^{2\tilde s}\right)^{k}u\right\|^{2}_{L^{2}(\mathbb R^3)}.
    \end{split}
    \end{equation}
    If $-2s<\gamma\le0$, we apply Lemma \ref{interpolation1}, then it follows that $2C_{3}\left\|\left(t\langle D\rangle^{2\tilde s}\right)^{k}u\right\|^{2}_{H^{2s-1}_{\gamma/2}}$ can be bounded by
     \begin{equation*}
         \frac18\left\|\left(t\langle D\rangle^{2\tilde s}\right)^{k}u\right\|^{2}_{H^{s}_{\gamma/2}(\mathbb R^3)}+\tilde C_{2}\left\|\left(t\langle D\rangle^{2\tilde s}\right)^{k}u\right\|^{2}_{L^{2}(\mathbb R^3)},
     \end{equation*}
    plugging it into the above inequalities, one has
    $$|I_{1}|\le\frac18\left\|\left(t\langle D\rangle^{2\tilde s}\right)^{k}u\right\|^{2}_{H^{s}_{\gamma/2}}+\tilde C_{2}\left\|\left(t\langle D\rangle^{2\tilde s}\right)^{k}u\right\|^{2}_{L^{2}},$$
    here $\tilde C_{2}$ depends on $C_{2}$ and $C_{3}$.

    For $I_{2}$, from Cauchy-Schwarz inequality, we have
    $$\left|I_{2}\right|\le2\left\|\langle \cdot\rangle^{-(\gamma/2+s)}\left[\left(t\langle D\rangle^{2\tilde s}\right)^{k}, \langle \cdot\rangle^{\gamma+2s}\right]u\right\|_{L^{2}}\left\|\left(t\langle D\rangle^{2\tilde s}\right)^{k}u\right\|_{2, \gamma/2+s},$$
    if $k=1$, then as the argument in Lemma \ref{lemma2.2},  for the commutator in $L^{2}$ on the right-hand side of the above inequality, we have
    $$\left\|\langle \cdot\rangle^{-(\gamma/2+s)}\left[t\langle D\rangle^{2\tilde s}, \langle \cdot\rangle^{\gamma+2s}\right]u\right\|_{L^{2}}\le C_{4}t\left\|u\right\|_{2, \gamma/2},$$
    with $C_{4}$ depends on $\gamma, s$, then assume that for $k-1$, 
    $$\left\|\langle \cdot\rangle^{-(\gamma/2+s)}\left[(t\langle D\rangle^{2\tilde s})^{k-1}, \langle \cdot\rangle^{\gamma+2s}\right]u\right\|_{L^{2}}\le\sum_{j=0}^{k-2}(C_{4}t)^{k-1-j}C_{k-1}^{j}\left\|(t\langle D\rangle^{2\tilde s})^{j}u\right\|_{2, \gamma/2+s}.$$
    For $k$, noting that 
    \begin{equation*}
    \begin{split}
         &\langle v\rangle^{-(\gamma/2+s)}\left[\left(t\langle D\rangle^{2\tilde s}\right)^{k}, \langle v\rangle^{\gamma+2s}\right]u=t\langle v\rangle^{-(\gamma/2+s)}\left[\left(t\langle D\rangle^{2\tilde s}\right)^{k-1}, \langle v\rangle^{\gamma+2s}\right]\langle D\rangle^{2\tilde s}u\\
         &\qquad+t\langle v\rangle^{\gamma/2+s}\left(t\langle D\rangle^{2\tilde s}\right)^{k-1}Au+t\langle v\rangle^{-(\gamma/2+s)}\left[\left(t\langle D\rangle^{2\tilde s}\right)^{k-1}, \langle v\rangle^{\gamma+2s}\right]Au,
    \end{split}
    \end{equation*}
    here $A(v, D_{v})=\langle v\rangle^{-(\gamma+2s)}[\langle D\rangle^{2\tilde s}, \langle v\rangle^{\gamma+2s}]$, as the argument in Lemma \ref{lemma2.2}, one has $A$ is a pseudo-differential operator of order $0$, hence 
    $$\left\|\langle \cdot\rangle^{\frac\gamma2+s}\left(t\langle D\rangle^{2\tilde s}\right)^{k-1}A\left(t\langle D\rangle^{2\tilde s}\right)^{-(k-1)}\left(t\langle D\rangle^{2\tilde s}\right)^{k-1}u\right\|_{L^{2}}\le C_{5}\left\|\left(t\langle D\rangle^{2\tilde s}\right)^{k-1}u\right\|_{2, \frac\gamma2+s}.$$
    Using the induction hypothesis, taking $C_{4}\ge C_{5}$, it follows that  
    \begin{equation}\label{4-4-k}
        \left\|\langle \cdot\rangle^{-(\frac\gamma2+s)}\left[(t\langle D\rangle^{2\tilde s})^{k}, \langle \cdot\rangle^{\gamma+2s}\right]u\right\|_{L^{2}}\le\sum_{j=0}^{k-1}(C_{4}t)^{k-j}C_{k}^{j}\left\|(t\langle D\rangle^{2\tilde s})^{j}u\right\|_{2, \frac\gamma2+s}.
    \end{equation}    
    Therefore from the Cauchy-Schwarz inequality, we can obtain that for any $t\in]0, T]$ 
    \begin{equation*}
    \begin{split}
        \left|I_{2}\right|\le2\left(\sum_{j=0}^{k-1}(C_{4}t)^{k-j}C_{k}^{j}\left\|(t\langle D\rangle^{2\tilde s})^{j}u\right\|_{2, \gamma/2+s}\right)^{2}+\frac12\left\|\left(t\langle D\rangle^{2\tilde s}\right)^{k}u\right\|^{2}_{2, \gamma/2+s}.
    \end{split}
    \end{equation*}
    Finally, we consider $I_{3}$, noticing that
    \begin{equation}\label{I3}
    \begin{split}
         I_{3}=2\left(\langle D\rangle^{-s}\langle v\rangle^{-\gamma/2}\left[\langle v\rangle^{\gamma}, (t\langle D\rangle^{2\tilde s})^{k}\right]\langle D\rangle^{2s}u, \langle D\rangle^{s}(\langle v\rangle^{\gamma/2}\left(t\langle D\rangle^{2\tilde s}\right)^{k}u)\right).
    \end{split}
    \end{equation}
    As the arguments in \eqref{4-4-k}, we have
    \begin{equation*}
    \begin{split}
         &\left\|\langle D\rangle^{-s}\langle \cdot\rangle^{-\gamma/2}\left[\langle \cdot\rangle^{\gamma}, (t\langle D\rangle^{2\tilde s})^{k}\right]\langle D\rangle^{2s}u\right\|_{L^{2}}\\
         &\le\sum_{j=0}^{k-1}(\tilde C_{4}t)^{k-j}C_{k}^{j}\left\|\langle D\rangle^{-s}\left(\langle \cdot\rangle^{\gamma/2}(t\langle D\rangle^{2\tilde s})^{j}\langle D\rangle^{2s}u\right)\right\|_{L^{2}},
    \end{split}
    \end{equation*}
    then from Lemma 2.2 of~\cite{H-4}, it follows that 
    \begin{equation*}
    \begin{split}
         &\left\|\langle D\rangle^{-s}\langle \cdot\rangle^{-\gamma/2}\left[\langle \cdot\rangle^{\gamma}, (t\langle D\rangle^{2\tilde s})^{k}\right]\langle D\rangle^{2s}u\right\|_{L^{2}}\\
         &\le\sum_{j=0}^{k-1}(\tilde C_{4}t)^{k-j}C_{k}^{j}\left\|\langle D\rangle^{-s}\left(\langle \cdot\rangle^{\gamma/2}\langle D\rangle^{2s}(t\langle D\rangle^{2\tilde s})^{j}u\right)\right\|_{L^{2}}\\
         &\le C_{5}\sum_{j=0}^{k-1}(\tilde C_{4}t)^{k-j}C_{k}^{j}\left\|(t\langle D\rangle^{2\tilde s})^{j}u\right\|_{H^{s}_{\gamma/2}},
    \end{split}
    \end{equation*}
    with $\tilde C_{4}$ and $C_{5}$ depends on $\gamma$, $s$. Plugging it into \eqref{I3}, we can get 
    $$\left|I_{3}\right|\le8\left(C_{5}\sum_{j=0}^{k-1}(\tilde C_{4}t)^{k-j}C_{k}^{j}\left\|(t\langle D\rangle^{2\tilde s})^{j}u\right\|_{H^{s}_{\gamma/2}}\right)^{2}+\frac18\left\|(t\langle D\rangle^{2\tilde s})^{k}u\right\|^{2}_{H^{s}_{\gamma/2}}.$$
 It remains to estimate $\left\|\langle D\rangle^{2\tilde sk}u\right\|_{L^{2}}$. Since $\gamma/2>-s>-1$, from Taylor formula, it follows that
    \begin{equation*}
    \begin{split}
        &\left|\langle v\rangle^{\gamma/2}-1\right|=\left|\sum_{j=1}^{3}\left(\int_{0}^{1}\partial_{j}\langle\theta v\rangle^{\gamma/2}d\theta v_{j}\right)\right|\le\sum_{j=1}^{3}\int_{0}^{1}\left|\partial_{j}\langle\theta v\rangle^{\gamma/2}\right|d\theta |v_{j}|\\
        &\le C_{\gamma}\int_{0}^{1}\langle\theta v\rangle^{\gamma/2}d\theta\le C_{\gamma}\langle v\rangle^{\gamma/2}\max\left\{1, \int_{0}^{1}\theta^{\gamma/2}d\theta\right\}\le\tilde C_{\gamma}\langle v\rangle^{\gamma/2},
    \end{split}
    \end{equation*}
    here $\tilde C_{\gamma}$ is the constant depend on $\gamma$, then by using Lemma \ref{lemma2.2}, it follows that
    \begin{equation*}
    \begin{split}
         \left\|\langle D\rangle^{\tilde s}\langle D\rangle^{2\tilde sp}u\right\|_{L^{2}}&\le\left\|\langle\cdot\rangle^{\gamma/2}\langle D\rangle^{\tilde s}\langle D\rangle^{2\tilde sp}u\right\|_{L^{2}}+\left\|(1-\langle\cdot\rangle^{\gamma/2})\langle D\rangle^{\tilde s}\langle D\rangle^{2\tilde sp}u\right\|_{L^{2}}\\
         &\le(\tilde C_{\gamma}+1)\left(\left\|\langle D\rangle^{2\tilde sp}u\right\|_{H^{\tilde s}_{\gamma/2}}+\left\|[\langle\cdot\rangle^{\gamma/2}, \langle D\rangle^{\tilde s}]\langle D\rangle^{2\tilde sp}u\right\|_{L^{2}}\right)\\
         &\le(\tilde C_{\gamma}+1)\left(\left\|\langle D\rangle^{2\tilde sp}u\right\|_{H^{\tilde s}_{\gamma/2}}+C_{1}\left\|\langle D\rangle^{2\tilde sp}u\right\|_{H^{\tilde s-1}_{\gamma/2}}\right)\\ 
         &\le(\tilde C_{\gamma}+1)(C_{1}+1)\left\|\langle D\rangle^{2\tilde sp}u\right\|_{H^{s}_{\gamma/2}}, \quad \forall p\in\mathbb N.
    \end{split}
    \end{equation*}
    Hence, using Cauchy-Schwarz inequality, we have
    \begin{equation*}
    \begin{split}
         &2kt^{2k-1}\left\|\langle D\rangle^{2\tilde sk}u\right\|^{2}_{L^{2}}=2kt^{2k-1}\left(\langle D\rangle^{\tilde s}\langle D\rangle^{2\tilde s(k-1)}u, \langle D\rangle^{\tilde s}\langle D\rangle^{2\tilde sk}u\right)\\
         &\le2k\left\|\langle D\rangle^{\tilde s}\left(t\langle D\rangle^{2\tilde s}\right)^{k-1}u\right\|_{L^{2}}\left\|\langle D\rangle^{\tilde s}\left(t\langle D\rangle^{2\tilde s}\right)^{k}u\right\|_{L^{2}}\\ 
         &\le2k(\tilde C_{\gamma}+1)^{2}(C_{1}+1)^{2}\left\|\left(t\langle D\rangle^{2\tilde s}\right)^{k-1}u\right\|_{H^{s}_{\gamma/2}}\left\|\left(t\langle D\rangle^{2\tilde s}\right)^{k}u\right\|_{H^{s}_{\gamma/2}}\\ 
         &\le\frac18\left\|\left(t\langle D\rangle^{2\tilde s}\right)^{k}u\right\|^{2}_{H^{s}_{\gamma/2}}+C_{6}k^{2}\left\|\left(t\langle D\rangle^{2\tilde s}\right)^{k-1}u\right\|^{2}_{H^{s}_{\gamma/2}},
    \end{split}
    \end{equation*}
    with $C_{6}$ depends on $\gamma, s$. Substituting these results into \eqref{dt}, one gets
    \begin{equation*}
    \begin{split}
        &\frac{d}{dt}\left\|\left(t\langle D\rangle^{2\tilde s}\right)^{k}u\right\|^{2}_{L^{2}}+\left\|\left(t\langle D\rangle^{2\tilde s}\right)^{k}u\right\|^{2}_{H^{s}_{\gamma/2}}+\left\|\left(t\langle D\rangle^{2\tilde s}\right)^{k}u\right\|^{2}_{2,\gamma/2+s}\\
        &\le\tilde C_{6}k^{2}\left\|\left(t\langle D\rangle^{2\tilde s}\right)^{k-1}u\right\|^{2}_{H^{s}_{\gamma/2}}+4\left(\sum_{j=0}^{k-1}(C_{4}t)^{k-j}C_{k}^{j}\left\|(t\langle D\rangle^{2\tilde s})^{j}u\right\|_{2, \gamma/2+s}\right)^{2}\\
        &\quad+16\left(C_{5}\sum_{j=0}^{k-1}(\tilde C_{4}t)^{k-j}C_{k}^{j}\left\|(t\langle D\rangle^{2\tilde s})^{j}u\right\|_{H^{s}_{\gamma/2}}\right)^{2}+2\left\|\left(t\langle D\rangle^{2\tilde s}\right)^{k}f\right\|^{2}_{L^{2}},
    \end{split}
    \end{equation*}
    with $\tilde C_{6}$ depends on $\gamma, s$. Integrating from~0~to $t$ gets, for all $0<t\le T$,
    \begin{equation*}
    \begin{split}
        &\left\|\left(t\langle D\rangle^{2\tilde s}\right)^{k}u\right\|^{2}_{L^{2}}+\int_{0}^{t}\left\|\left(\tau\langle D\rangle^{2\tilde s}\right)^{k}u\right\|^{2}_{H^{s}_{\gamma/2}}d\tau+\int_{0}^{t}\left\|\left(\tau\langle D\rangle^{2\tilde s}\right)^{k}u\right\|^{2}_{2, \gamma/2+s}d\tau\\
        &\le\tilde C_{6}k^{2}\int_{0}^{t}\left\|\left(\tau\langle D\rangle^{2\tilde s}\right)^{k-1}u\right\|^{2}_{H^{s}_{\gamma/2}}d\tau+2\int_{0}^{t}\left\|\left(\tau\langle D\rangle^{2\tilde s}\right)^{k}f\right\|^{2}_{L^{2}}d\tau\\
        &\quad+4\int_{0}^{t}\left(\sum_{j=0}^{k-1}(C_{4}\tau)^{k-j}C_{k}^{j}\left\|(\tau\langle D\rangle^{2\tilde s})^{j}u\right\|_{2, \gamma/2+s}\right)^{2}d\tau\\
        &\quad+16\int_{0}^{t}\left(C_{5}\sum_{j=0}^{k-1}(\tilde C_{4}\tau)^{k-j}C_{k}^{j}\left\|(\tau\langle D\rangle^{2\tilde s})^{j}u\right\|_{H^{s}_{\gamma/2}}\right)^{2}d\tau.
    \end{split}
    \end{equation*}
    By using Minkowski inequality, taking $B_{1}\ge(C_{4}T)^{2}+1$, from \eqref{m}, one has 
    \begin{equation*}
    \begin{split}
    &\int_{0}^{t}\left(\sum_{j=0}^{k-1}(C_{4}\tau)^{k-j}C_{k}^{j}\left\|(\tau\langle D\rangle^{2\tilde s})^{j}u\right\|_{2, \gamma/2+s}\right)^{2}d\tau\\
    &\le\left(\sum_{j=0}^{k-1}(C_{4}T)^{k-j}C_{k}^{j}\left(\int_{0}^{t}\left\|(\tau\langle D\rangle^{2\tilde s})^{j}u\right\|^{2}_{2, \gamma/2+s}d\tau\right)^{\frac12}\right)^{2}\\
    &\le\left(\sum_{j=0}^{k-2}\frac{(C_{4}T)^{k-j}B_{1}^{j+1}k!}{(k-j)!}+C_{4}TB_{1}^{k}k!\right)^{2}\le\left((3+C_{4}T)B_{1}^{k}k!\right)^{2},
    \end{split}
    \end{equation*}
    similarly, taking $B_{1}\ge(\tilde C_{4}T)^{2}+1$, from \eqref{m}, we have 
    $$\int_{0}^{t}\left(C_{5}\sum_{j=0}^{k-1}(\tilde C_{4}\tau)^{k-j}C_{k}^{j}\left\|(\tau\langle D\rangle^{2\tilde s})^{j}u\right\|_{H^{s}_{\gamma/2}}\right)^{2}d\tau\le\left(C_{5}(3+\tilde C_{4}T)B_{1}^{k}k!\right)^{2}.$$
    Then using \eqref{f} and \eqref{m}, it follows that 
    \begin{equation*}
    \begin{split}
        &\left\|\left(t\langle D\rangle^{2\tilde s}\right)^{k}u\right\|^{2}_{L^{2}}+\int_{0}^{t}\left\|\left(\tau\langle D\rangle^{2\tilde s}\right)^{k}u\right\|^{2}_{H^{s}_{\gamma/2}}d\tau+\int_{0}^{t}\left\|\left(\tau\langle D\rangle^{2\tilde s}\right)^{k}u\right\|^{2}_{2, \gamma/2+s}d\tau\\
        &\le2T(T^{k}A^{k+1}k!)^{2}+16\left(\tilde C_{6}+(C_{5})^{2}(3+\tilde C_{5}T)^{2}\right)(B_{1}^{k}k!)^{2},
    \end{split}
    \end{equation*}
    here $\tilde C_{5}=\max\{C_{4}, \tilde C_{4}\}$. Finally, taking
    $$B_{1}\ge\max\left\{B_{0}, TA+1, (\tilde C_{5}T)^{2}+1, 4\sqrt{\tilde C_{6}+(C_{5})^{2}(3+\tilde C_{5}T)^{2}+TA^{2}}+1, \right\},$$
    we get for any $0<t\le T$,
    \begin{equation*}
    \begin{split}
        &\left\|\left(t\langle D\rangle^{2\tilde s}\right)^{k}u\right\|^{2}_{L^{2}}+\int_{0}^{t}\left\|\left(\tau\langle D\rangle^{2\tilde s}\right)^{k}u\right\|^{2}_{H^{s}_{\gamma/2}}d\tau\\
        &\quad+\int_{0}^{t}\left\|\left(\tau\langle D\rangle^{2\tilde s}\right)^{k}u\right\|^{2}_{2,\gamma/2+s}d\tau\le\left(B_{1}^{k+1}k!\right)^{2}.
    \end{split}
    \end{equation*}
\end{proof}

Next, we construct the following estimate to prove Theorem \ref{thm2}.

\begin{prop}\label{prop 5.1}
    For $0<s<1$ and $\gamma+2s\ge0$. Let $u$ be the solution of Cauchy problem \eqref{1-2}, $u_0\in L^2(\mathbb R^3)$ and $e^{t\langle v\rangle^{\gamma/2+s}}f\in L^2(\mathbb R^3)$. Then there exists a constant $B_{2}>0$ such that for all $k\in\mathbb N$, $0<t\le T$,
    \begin{equation}\label{5-1}
    \begin{split}
        &\left\|\left(t\langle \cdot\rangle^{\gamma/2+s}\right)^{k}u\right\|^2_{(L^{\infty}]0,T];L^2)}+\int_0^T\left\|\left(t\langle \cdot\rangle^{\gamma/2+s}\right)^ku\right\|^2_{H^s_{\gamma/2}}dt\\
        &\qquad+\int_0^T\left\|\left(t\langle \cdot\rangle^{\gamma/2+s}\right)^ku\right\|^2_{2,\gamma/2+s}dt\le\left(B_{2}^{k+1}k!\right)^2.
    \end{split}
    \end{equation}
\end{prop}
\begin{proof}
    From \eqref{1-2}, we have
    \begin{equation}\label{5-1-1}
    \begin{split}
         &\frac12\frac{d}{dt}\left\|\left(t\langle\cdot\rangle^{\gamma/2+s}\right)^{k}u\right\|^{2}_{L^{2}}+\left\|\left(t\langle\cdot\rangle^{\gamma/2+s}\right)^{k}u\right\|^{2}_{2,\gamma/2+s}\\
         &\quad+\left((1-\Delta)^{s}\left((t\langle v\rangle^{\gamma/2+s})^{k}\langle v\rangle^{\gamma/2}u\right), \langle v\rangle^{\gamma/2}\left(t\langle v\rangle^{\gamma/2+s}\right)^{k}u\right)\\
         &=kt^{2k-1}\left\|\langle\cdot\rangle^{(\gamma/2+s)k}u\right\|^{2}_{L^{2}}+\left(\left(t\langle v\rangle^{\gamma/2+s}\right)^{k}f, \left(t\langle v\rangle^{\gamma/2+s}\right)^{k}u\right)\\
         &\quad+\left(\left[(1-\Delta)^{s}, \left(t\langle v\rangle^{\gamma/2+s}\right)^{k}\langle v\rangle^{\gamma/2}\right]u, \langle v\rangle^{\gamma/2}\left(t\langle v\rangle^{\gamma/2+s}\right)^{k}u\right)\\
         &=kt^{2k-1}\left\|\langle\cdot\rangle^{(\gamma/2+s)k}u\right\|^{2}_{L^{2}}+R_{1}+R_{2}.
    \end{split}
    \end{equation}

    We show \eqref{5-1} holds by induction on the index $k$. For $k=0$, it is enough to take in \eqref{3-1}. Assume $k\ge1$ and \eqref{5-1} is true for $0\le m\le k-1$,
    \begin{equation}\label{m1}
    \begin{split}
        &\left\|\left(t\langle \cdot\rangle^{\gamma/2+s}\right)^{m}u\right\|^2_{(L^{\infty}]0,T];L^2)}+\int_0^T\left\|\left(t\langle \cdot\rangle^{\gamma/2+s}\right)^mu\right\|^2_{H^s_{\gamma/2}}dt\\
        &\qquad+\int_0^T\left\|\left(t\langle \cdot\rangle^{\gamma/2+s}\right)^mu\right\|^2_{2,\gamma/2+s}dt\le\left(B_{2}^{m+1}m!\right)^2.
    \end{split}
    \end{equation}

    Now, we prove that \eqref{5-1} holds true for $m=k$. By using the G${\rm\mathring{a}}$rding inequality, it follows that
    \begin{equation*}
    \begin{split}
        &\left((1-\Delta)^{s}\left((t\langle v\rangle^{\gamma/2+s})^{k}\langle v\rangle^{\gamma/2}u\right), \langle v\rangle^{\gamma/2}\left(t\langle v\rangle^{\gamma/2+s}\right)^{k}u\right)\\
        &\ge\frac12\left\|\left(t\langle\cdot\rangle^{\gamma/2+s}\right)^{k}u\right\|^{2}_{H^{s}_{\gamma/2}}-C_{0}\left\|\left(t\langle\cdot\rangle^{\gamma/2+s}\right)^{k}u\right\|^{2}_{2,\gamma/2},
    \end{split}
    \end{equation*}
    as the argument in the Lemma \ref{lemma3.1}, applying Lemma \ref{interpolation} if $\gamma>0$ and applying Lemma \ref{interpolation1} if $-2s<\gamma\le0$, respectively, we get
    $$C_{0}\left\|\left(t\langle\cdot\rangle^{\gamma/2+s}\right)^{k}u\right\|^{2}_{2,\gamma/2}\le\frac18\left\|\left(t\langle\cdot\rangle^{\gamma/2+s}\right)^{k}u\right\|^{2}_{H^{s}_{\gamma/2}}+\tilde C_{0}\left\|\left(t\langle\cdot\rangle^{\gamma/2+s}\right)^{k}u\right\|^{2}_{L^{2}}.$$
    For the term $R_{1}$, by using Cauchy-Schwarz inequality, one has
    $$|R_{1}|\le\frac12\left\|\left(t\langle\cdot\rangle^{\gamma/2+s}\right)^{k}f\right\|^{2}_{L^{2}}+\frac12\left\|\left(t\langle\cdot\rangle^{\gamma/2+s}\right)^{k}u\right\|^{2}_{L^{2}}.$$
    For $R_{2}$, from the Plancherel theorem, as the arguments in \eqref{4-4-k}, it follows that 
    $$\left\|\left[\langle D\rangle^{2s}, \left(t\langle \cdot\rangle^{\gamma/2+s}\right)^{k}\langle \cdot\rangle^{\gamma/2}\right]u\right\|_{L^{2}}\le\sum_{j=1}^{k}(C_{7}t)^{k-(j-1)}C_{k}^{j}\left\|\left(t\langle\cdot\rangle^{\gamma/2+s}\right)^{j}u\right\|_{H^{2s-1}_{\gamma/2}}.$$
    Hence, using Cauchy-Schwarz inequality, it follows that for $0<s\le1/2$
    $$|R_{2}|\le\sum_{j=1}^{k}(C_{7}t)^{k-(j-1)}C_{k}^{j}\left\|\left(t\langle\cdot\rangle^{\gamma/2+s}\right)^{j}u\right\|_{2, \gamma/2}\left\|\left(t\langle\cdot\rangle^{\gamma/2+s}\right)^{k}u\right\|_{2, \gamma/2},$$
    and for $1/2<s<1$, we have
     $$|R_{2}|\le\sum_{j=1}^{k}(C_{7}t)^{k-(j-1)}C_{k}^{j}\left\|\left(t\langle\cdot\rangle^{\gamma/2+s}\right)^{j}u\right\|_{H^{2s-1}_{\gamma/2}}\left\|\left(t\langle\cdot\rangle^{\gamma/2+s}\right)^{k}u\right\|^{2}_{H^{2s-1}_{\gamma/2}}.$$
     Then for both two cases, using Lemma \ref{interpolation} for $\gamma>0$ and using Lemma \ref{interpolation1} for $-2s<\gamma\le0$, we obtain that for all $0<s<1$
     \begin{equation*}
     \begin{split}
         |R_{2}|&\le\frac18\left\|\left(t\langle\cdot\rangle^{\gamma/2+s}\right)^{k}u\right\|^{2}_{H^{s}_{\gamma/2}}+\tilde C_{7}\left\|\left(t\langle\cdot\rangle^{\gamma/2+s}\right)^{k}u\right\|^{2}_{L^{2}}\\
         &\quad+\left(\sum_{j=1}^{k-1}(C_{7}t)^{k-(j-1)}C_{k}^{j}\left\|\left(t\langle\cdot\rangle^{\gamma/2+s}\right)^{j}u\right\|_{H^{s}_{\gamma/2}}\right)^{2},
     \end{split}
    \end{equation*}
     with $\tilde C_{7}$ depends on $\gamma$ and $s$.

     Plugging these results back into \eqref{5-1-1}, one gets
     \begin{equation*}
     \begin{split}
         &\frac{d}{dt}\left\|\left(t\langle\cdot\rangle^{\gamma/2+s}\right)^{k}u\right\|^{2}_{L^{2}}+\left\|\left(t\langle\cdot\rangle^{\gamma/2+s}\right)^{k}u\right\|^{2}_{2,\gamma/2+s}+\left\|\left(t\langle\cdot\rangle^{\gamma/2+s}\right)^{k}u\right\|^{2}_{H^{s}_{\gamma/2}}\\
         &\le4kt\left\|\left(t\langle\cdot\rangle^{(\gamma/2+s)}\right)^{k-1}u\right\|^{2}_{2, \gamma/2+s}+C_{8}t^{2}\left\|\left(t\langle\cdot\rangle^{(\gamma/2+s)}\right)^{k-1}u\right\|^{2}_{2, \gamma/2+s}\\
         &\quad+2\left\|\left(t\langle\cdot\rangle^{\gamma/2+s}\right)^{k}f\right\|^{2}_{L^{2}}+4\left(\sum_{j=1}^{k-1}(C_{7}t)^{k-(j-1)}C_{k}^{j}\left\|\left(t\langle\cdot\rangle^{\gamma/2+s}\right)^{j}u\right\|_{H^{s}_{\gamma/2}}\right)^{2},
    \end{split}
    \end{equation*}
    with $C_{8}=4(\tilde C_{0}+\tilde C_{7}+1)$. Integrating from 0 to $t$, using Minkowski inequality and \eqref{m1}, taking $B_{2}\ge(C_{7}T)^{3}+1$, we obtain for all $t\in[0, T]$
    \begin{equation*}
     \begin{split}
          &\left\|\left(t\langle\cdot\rangle^{\gamma/2+s}\right)^{k}u\right\|^{2}_{L^{2}}+\int_{0}^{t}\left\|\left(\tau\langle\cdot\rangle^{\gamma/2+s}\right)^{k}u\right\|^{2}_{H^{s}_{\gamma/2}}d\tau\\
          &\quad+\int_{0}^{t}\left\|\left(\tau\langle\cdot\rangle^{\gamma/2+s}\right)^{k}u\right\|^{2}_{2,\gamma/2+s}d\tau\\
          &\le\left(4kT+C_{8}T^{2}\right)\left(B_{2}^{k}(k-1)!\right)^{2}+4\left((4+(C_{7}T)^{2})B_{2}^{k}k!\right)^{2}\\
          &\quad+2\int_{0}^{t}\left\|\left(\tau\langle\cdot\rangle^{\gamma/2+s}\right)^{k}f\right\|^{2}_{L^{2}}d\tau.
     \end{split}
    \end{equation*}
    Since $e^{t\langle v\rangle^{\gamma/2+s}}f\in L^2(\mathbb R^3)$, by using Taylor Formula it follows that there exists a constant $\tilde A>0$ such that
    $$\left\|\left(t\langle \cdot\rangle^{\gamma/2+s}\right)^{k}f\right\|_{L^{2}}\le t^{k}\tilde A^{k+1}k!,$$
    and therefore taking
     $$B_{2}\ge\max\left\{2\sqrt{T+C_{8}T^{2}+(4+(C_{7}T)^{2})^{2}+T\tilde A^{2}}+1, (C_{7}T)^{3}+1, B_{0}, \tilde AT+1\right\},$$
     we get for all $0<t\le T$,
     \begin{equation*}
    \begin{split}
        &\left\|\left(t\langle \cdot\rangle^{\gamma/2+s}\right)^{k}u\right\|^2_{(L^{\infty}]0,T];L^2)}+\int_0^T\left\|\left(t\langle \cdot\rangle^{\gamma/2+s}\right)^ku\right\|^2_{H^s_{\gamma/2}}dt\\
        &\qquad+\int_0^T\left\|\left(t\langle \cdot\rangle^{\gamma/2+s}\right)^ku\right\|^2_{2,\gamma/2+s}dt\le\left(B_{2}^{k+1}k!\right)^2.
    \end{split}
    \end{equation*}
\end{proof}


\bigskip
\noindent {\bf Acknowledgements.}
This work was supported by the NSFC (No.12031006) and the Fundamental
Research Funds for the Central Universities of China.

\end{document}